\documentclass[12pt,a4paper]{amsart}
\usepackage{amsfonts}
\usepackage{amsthm}
\usepackage{amsmath}
\usepackage{amscd}
\usepackage[latin2]{inputenc}
\usepackage{t1enc}
\usepackage[mathscr]{eucal}
\usepackage{indentfirst}
\usepackage{graphicx}
\usepackage{graphics}
\usepackage{pict2e}
\usepackage{epic}
\numberwithin{equation}{section}
\usepackage[margin=2.9cm]{geometry}
\usepackage{epstopdf}

\usepackage{thmtools}
\usepackage{thm-restate}

\usepackage{hyperref}

\usepackage{cleveref}

\theoremstyle{plain}
\newtheorem{Th}{Theorem}[section]
\newtheorem{Lemma}[Th]{Lemma}

\theoremstyle{definition}
\newtheorem{Def}[Th]{Definition}

\newtheorem{Rem}[Th]{Remark}
\newtheorem{?}[Th]{Problem}
\newtheorem{Ex}[Th]{Example}

\begin{document}

\title{Symplectic Matroids, Circuits, and Signed Graphs}

\author{Zhexiu Tu}

\address{University of the South\\ Department of Math and Computer Science\\
Sewanee, TN37355} 

\email{zhtu@sewanee.edu}



\begin{abstract} One generalization of ordinary matroids is symplectic matroids. While symplectic matroids were initially defined by their collections of bases, there has been no cryptomorphic definition of symplectic matroids in terms of circuits. We give a definition of symplectic matroids by collections of circuits. As an application, we construct a class of examples of symplectic matroids from graphs in terms of circuits. 
\end{abstract}

\maketitle

\section{Introduction}

A matroid is a combinatorial structure that generalizes the notion of linear independence in vector spaces. There are many textbooks on this subject. We refer the readers to \cite{oxley} for more background on matroids. There are different cryptomorphic characterizations of matroids, for example, in terms of bases, circuits, flats, etc.. Below we list a matroid definition in terms of circuits.

\begin{Def}
A finite \textit{matroid} $M$ is a pair $(E,\mathcal{C})$, where $E$ is a finite set (called the ground set) and $\mathcal {C}$ is a family of subsets of $E$ (called the circuits) with the following properties:
\begin{enumerate}
\item[(C1)] $\emptyset \notin \mathcal{C}$.
\item[(C2)] $C_1 , C_2 \in \mathcal{C}$ and $C_2 \subseteq C_1$ implies $C_2 = C_1$.
\item[(C3)] $C_1, C_2 \in \mathcal{C}$ with $C_1 \neq C_2$ and $e \in C_1 \cap C_2$ implies there exists some $C_3 \in \mathcal{C}$ such that $C_3 \subset (C_1 \cup C_2 ) - \{ e \}$.
\end{enumerate}
\end{Def}

We refer to matroids as \textit{ordinary matroids}, to distinguish them from different generalizations of matroids, such as \textit{symplectic matroids}. Symplectic matroids are obtained when we replace the symmetric group with the \textit{hyperoctahedral group}, a  group of symmetries of the $n$-cube $[-1, 1]^n$. Geometrically, symplectic matroids are related to the vector spaces endowed with bilinear forms, although in a way different from the way ordinary matroids are related to vector spaces. Symplectic matroids are a generalization of the following matroids that are all equivalent, $\Delta$-matroids \cite{bouch}, metroids \cite{bouch}, or 2-matroids \cite{bouch}. 

Symplectic matroids were defined in \cite{borov} by Borovik, Gelfand and White using the maximality property of bases. In 2003, T. Chow defined symplectic matroids in terms of independent sets, and proved the equivalence between the two definitions in \cite{chow}. In \cite{chow}, Chow posed a  complicated exchange property on independent sets, and proposed a conjectural exchange property on the collection of bases. We know of no progress toward defining symplectic matroids using any other axiomatizations similar to those of ordinary matroids.

In \cite{borov}, a special type of symplectic matroids, called \textit{Lagrangian matroids}, which turn out to be equivalent to $\Delta$-matroids, was studied. Borovik, Gelfand and White provided the circuit axiomatizations of Lagrangian matroids and proved the equivalence between the definitions. However, Lagrangian matroids are just a special case of all symplectic matroids. At present there are no circuit axiomatizations of symplectic matroids.

In this paper, we define symplectic matroids in terms of circuits. Some of these axioms resemble circuit axioms for ordinary matroids, including the circuit elimination axiom. We prove the equivalence between our definition and the definition by Borovik et al in \cite{borov}.  As an application of this result, we show how every finite undirected multigraph gives rise to a symplectic matroid in terms of circuits. 

It is worth mentioning that in \cite{borov} a special type of symplectic matroids, called the \textit{Lagrangian matroids}, was studied. Borovik, Gelfand and White provided and proved the equivalence between the definitions of Lagrangian matroids in terms of bases and circuits. Lagrangian matroids are a class of Coxeter matroids where we let $W = BC_n \cong S_2 \wr S_n$ and $P = S_n$. Hence they are the symplectic matroids where the size of each basis matches the size of the symplectic ground set. In other words, they are the full-rank symplectic matroids. Lagrange matroids are in the names of symmetric matroids \cite{bouch30}, $\Delta$-matroids \cite{bouch30}, metroids \cite{bouch35}, or 2-matroids \cite{bouch34}. 

The structure of this paper is as follows. In Section~\ref{sec2}, we give basic definitions and terms that we will use in our proofs. In Section~\ref{sec3}, we give an alternative definition or axiomatization of a symplectic matroid in terms of circuits (Theorem~\ref{thm:main}), which is our main theorem. In Section~\ref{sec4}, we show that symplectic matroids always satisfy the circuit axioms that we have defined in Section~\ref{sec3}. In Section~\ref{sec5}, we go backwards and show that out circuit axioms guarantee symplectic matroids. It then suffices to prove Theorem~\ref{thm:main}. In Section~\ref{sec6}, we apply Theorem~\ref{thm:main} and construct a class of examples of symplectic matroids from graphs in terms of circuits.

\section{Background and definitions}
\label{sec2}
In this section we give the basic definitions of symplectic matroids. Let 
\[ [n] = \{1,2,\ldots,n \} \textrm{ and }  [n]^* = \{1^*,2^*,\ldots,n^* \} \]
where the map $*: [n] \to [n]^*$ is defined by $i \mapsto i^*$ and $*: [n]^* \to [n]$ is defined by $i^* \mapsto i$. We apply $*$ to sets and collections of sets, for example $C^*$ and $\mathcal{C}^*$. Let 
\[ E_{\pm n} : =  [n] \cup [n]^*\]
be the new ground set. Thus $i^{**} = i$ signifies that $i$ is an involutive permutation of $E_{\pm n}$. That is why sometimes we write $i^*$ as $-i$ and $E_{\pm n}$ can be thought of as a set equivalent to $\{ -n, -(n-1) , \ldots, -1 , 1, 2, \ldots , n \}$. We say a set $S$ is \textit{admissible} if $S \cap S^* = \emptyset$. A permutation $\omega$ of $E_{\pm n}$ is \textit{admissible} if $\omega (x^*) = \omega (x)^*$ for all $x \in E_{\pm n}$. An ordering $<$ on $E_{\pm n}$ is \textit{admissible} if and only if $<$ is a linear ordering and from $i < j$ it follows that $j^* < i^*$. Denote by $E_k$ the collection of all admissible $k$-subsets in $E_{\pm n}$, for $k < 2n$. If $<$ is an arbitrary linear ordering on $E_{\pm n}$, it induces the partial ordering (which we also denote by the same symbol $<$) on $E_k$: if $A, B \in E_k$ and
\[ A := \{ a_1 < a_2 < \ldots < a_k \}  \textrm{ and } B := \{b_1 < b_2 < \ldots < b_k \},\] 
we set $A \leq B$ if
\[ a_1 \leq b_1 , \,\, a_2 \leq b_2, \,\, \ldots , \,\, a_k \leq b_k. \]

We can visualize an admissible ordering as a signed permutation $\sigma$ of $[n]$ followed by the negative of the reversal of $\sigma$. For example, when $n = 3$
\[ 1 < 3 < 2^* < 2 < 3^* < 1^* \]
is one admissible ordering. 

\begin{Def}
If $\mathcal{B}$ is a non-empty family
of equi-numerous admissible subsets of $E_{\pm n}$ with the property that for every admissible ordering $<$ of $E_{\pm n}$, the collection $\mathcal{B}$ always contains a unique maximal element, then $M = (E_{\pm n}; \mathcal{B})$ is a symplectic matroid, and $\mathcal{B}$ is called the collection of \textit{bases} of $M$.
\end{Def}

Below is an example of a non-symplectic matroid.

\begin{Ex}
Let $n = 3$ and $k = 2$, and let $\mathcal{B} = \{12, 2^*3, 13 \}$, where we use our abbreviated notation by listing $\{a, b\}$ as $ab$. Consider the admissible ordering $1 < 3 < 2^* < 2 < 3^* < 1^*$. Then $12$ and $2^* 3$ are incomparable in the induced ordering on $E_2$, and both are larger than $13$, hence $\mathcal{B}$ cannot be a symplectic matroid.
\end{Ex}

\section{Circuits}
\label{sec3}

Let $M = (E_{\pm n}; \mathcal{B})$ be a symplectic matroid, where $\mathcal{B}$ is the collection of bases of $M$. Let $\mathcal{C}$ be the collection of minimal admissible subsets of $E_{\pm n}$ not contained in any member of $\mathcal{B}$. That collection of subsets $\mathcal{C}$ is called the collection of \textit{circuits} of $M$. An admissible set containing no circuits as its subset is called an \textit{independent set}. Otherwise, it is \textit{dependent}.

We let $A \Delta B$ be the \textit{symmetric difference} between two sets $A$ and $B$ defined by $A \Delta B = A \cup B - A \cap B$. We give an important definition of the term \textit{span}.
\begin{Def}
Let $\mathcal{C}$ be a collection of admissible subsets of $E_{\pm n}$. Then an admissible set $P$ spans $x \in E_{\pm n}$ if there exist some $J \in \mathcal{C}$ such that $J - P = \{x\}$.
\end{Def}

A characterization of $\mathcal{C}$ could be used as an alternative definition or axiomatization of
a symplectic matroid. This is precisely what the following theorem provides, followed by an example.

\begin{restatable}{thm}{main}
\label{thm:main}
Let $\mathcal{B}$ be the collection of bases of a symplectic matroid. Let $\mathcal{C}$ be the collection of minimal admissible subsets of $E_{\pm n}$ not contained in any member of $\mathcal{B}$.  Then $\mathcal{C}$ satisfies the following four properties.
\begin{enumerate}
\item[(SC1)] $\emptyset \notin \mathcal{C}$.
\item[(SC2)] If $C_1,C_2 \in \mathcal{C}$ with $C_1 \subseteq C_2$, then $C_1 = C_2$.
\item[(SC3)] If $C_1,C_2 \in \mathcal{C}$ with $C_1 \neq C_2$, $x \in C_1 \cap C_2$ and $C_1 \cup C_2$ is admissible, then there exists some $C \in \mathcal{C}$ with $C \subseteq (C_1 \cup C_2)-\{x\}$.
\item[(SC4)] Let $P$ be an admissible subset of $E_{\pm n}$ and $B \in \mathcal{B}$. If $|P| < |B|$,  $P$ does not span $E_{\pm n}-P \cup P^*$.
\end{enumerate}
Conversely, let $\mathcal{C}$ be a collection of admissible subsets of $E_{\pm n}$, and let $\mathcal{B}$ be the collection of maximal admissible subsets of $E_{\pm n}$ not containing members of $\mathcal{C}$.  If $\mathcal{C}$ satisfies (SC1) - (SC4), then $\mathcal{B}$ is the collection of bases of a symplectic matroid.
\end{restatable}

\begin{Rem}
(SC1), (SC2) and (SC3) resemble the circuit axioms of ordinary matroids. However, they don't suffice to guarantee the equi-cardinality of bases of symplectic matroids. (SC4) guarantees  the equi-cardinality of bases.
\end{Rem}

\begin{Ex}
Let $\mathcal{B} = \{\{1, 2, 3\}, \{1^*, 2^*, 3\}, \{1, 3, 4\}, \{2^*, 3, 4\} \}$. Then 
\[ \{ \mathcal{C} = \{3^*\}, \{4^*\}, \{1^*, 2\}, \{1, 2^*\}, \{1^*, 4\}, \{2, 4\} \}\]
is the collection of minimal admissible subsets not contained in any member of $\mathcal{B}$. Meantime, $\mathcal{B}$ is the collection of maximal admissible subsets not containing members of $\mathcal{C}$.

We can check that $\mathcal{C}$ satisfies (SC1), (SC2) and (SC3) without much obstacle. For any admissible set $P$ with $|P| = 4$, it contains some $C \in \mathcal{C}$. For any admissible set $Q = \{a ,  b \}$ where $a,b \in [4] \cup [4]^*$, $Q$ doesn't span $E_{\pm 4} - Q \cup Q^*$.
\end{Ex}

\section{Symplectic matroids satisfying circuit axioms}
\label{sec4}

Throughout this section, $\mathcal{B}$ is the collection of bases of a symplectic matroid $M$, and $\mathcal{C}$ is the collection of minimal admissible subsets of $E_{\pm n}$ not contained in any member of $\mathcal{B}$.
\begin{Lemma}
\label{lemma1}
Let $B \in \mathcal{B}$, and some $x \notin B$ such that $B \cup \{ x \}$ is admissible. Then there exists a unique circuit $C \subseteq B \cup \{x\}$ where $C$ is given by 
\[C=\{x\} \cup \{b \in B \mid B \cup \{x\} - \{b\} \in \mathcal{B}\}.\]
\end{Lemma}
\begin{proof}
Let $B \in \mathcal{B}$ and $x \notin B$ such that $B \cup \{x\}$ is admissible. Then $|B \cup \{x\}| > |B|$. Therefore, $B \cup \{x\}$ is dependent, which means $B \cup \{x\}$ contains a circuit. Since $\{b \in B \mid B \cup \{x\} - \{b\} \in \mathcal{B}\} \subseteq B$ and $B \cup \{x\}$ is admissible, so $C$ is definitely admissible. 

The expression of the unique circuit $C$ and its proof for symplectic matroids is the same as those for ordinary matroids, which can be found in various papers or textbooks, for example, in \cite{minieka}.
\end{proof}

\begin{Lemma}
\label{lemma2}
Let $C_1$ and $C_2$ be two distinct circuits of M, $C_1 \cup C_2$ be admissible and $x \in C_1 \cap C_2$. Then for every $c\in C_1 \Delta C_2$, there exists some $C_c \in \mathcal{C}$ such that $c \in C_c \subseteq C_1 \cup C_2 - \{x\}$.
\end{Lemma}
\begin{proof}
Suppose $C_1 \cup C_2 - \{x\}$ is independent. Then $C_1 \cup C_2 - \{x\} \subseteq B$. We know $x \notin B$, otherwise $C_1 \subseteq B$. Hence, $C_1,C_2 \subseteq C_1 \cup C_2 \subseteq B \cup \{x\}$. $B \cup \{x\}$ is dependent because B is a basis, and $B \cup \{x\}$ is admissible. Thus $B \cup \{x\}$ contains a unique circuit by Lemma~\ref{lemma1}. That contradicts $C_1$ and $C_2$ being distinct. Thus $C_1 \cup C_2 - \{x\}$ is dependent.

Since we suppose $C_1 \cup C_2$ is admissible, we show the existence of such a circuit $C_c$. This proof resembles that in \cite{borov}. We proceed by induction on $|C_1 \cup C_2|$. For the base step of induction, consider $C_1 = \{c_1,x\}$ and $C_2 = \{c_2,x\}$. Then $C=\{c_1,c_2\}=C_1 \Delta C_2$ must be a circuit. For the inductive step, let $c \in C_2 - \ C_1$ without the loss of generality. We have shown that there exists a circuit $C \subseteq (C_1 \cup C_2) - \{x\}$. Suppose $c \notin C$. Since $C \not \subseteq C_2$, there exists some $y \in (C \cap C_1) - C_2$. We notice $x \in C_1 - C$, but $c \notin C \cup C_1$. Thus, $C \cup C_1 \subset C_1 \cup C_2$ and we can apply the induction hypothesis to $C$, $C_1$, and $x, y$ to find a circuit $C_3$ with $x \in C_3 \subseteq (C \cup C_1) - \{y\}$. Since $y \notin C_2$ and $y \notin C_3$, we have $C_3 \cup C_2 \subset C_1 \cup C_2$. However, $x \in C_2 \cap C_3$ and $c \in C_2 - C_3$. Thus, by applying the induction hypothesis again, we get a circuit $C_c$ with $c \in C_c \subseteq (C_3 \cup C_2) - \{x\} \subseteq (C_1 \cup C_2) - \{x\}$.
\end{proof} 

\begin{Th}
\label{theorem1}
Let $P$ be an admissible subset of $E_{\pm n}$ and $B \in \mathcal{B}$. If $|P| < |B|$,  $P$ does not span $E_{\pm n}-P \cup P^*$.
\end{Th}
\begin{proof}
Suppose there exists some $P$ such that $|P| < |B| = k$ and $P$ is the minimal set that spans $E_{\pm n}-P \cup P^*$, which means no subset $P_0$ of $P$ spans $E_{\pm n}-P_0 \cup P_0^*$. Without the loss of generality, suppose $P = \{ 1 , 2, \ldots, k-1 \}$. Hence $P$ spans every element in $\{ k, k+1,\ldots,n \} \cup \{ k, k+1,\ldots,n \}^*$. Thus there exist some $J_{k+j} \in \mathcal{C}$ such that
\[ J_{k+j} - P  = \{ k+j \}  \]
for all $j = 0 , \ldots, n-k$, and $J_{(k+j)^*} \in \mathcal{C}$ such that
\[ J_{(k+j)^*} - P  = \{ (k+j)^* \} \]
for all $j = 0 , \ldots, n-k$. However, $P$ cannot be independent because $P \cup \{x \}$ is always dependent for any $x \in E_{\pm n}-P \cup P^*$, which makes $P$ a basis of size $k-1$, a contradiction. So $P$ is dependent. Thus $P \not \subseteq J_{k+j}$ nor $P \not \subseteq J_{(k+j)^*}$ for all $j$.  

Suppose $P$ is a circuit. (The proof when $P$ contains a circuit is similar.)  There exists some $z \in P$ such that $z \notin J_{n^*}$. Let $S : = P - \{ z \}$.  Then $J_{n^*} - \{ n^* \} \subseteq S$. Thus $S \cup \{ n^* \}$ is dependent because $J_{n^*} \subseteq S \cup \{ n^* \}$. For any $x \in E_{\pm n}-P \cup P^*$ and $x \neq n$, if $z \in J_x$, then by Lemma~\ref{lemma2}, there exists some $C \subseteq J_x \cup P - \{ z \}$, which means $S \cup \{ x \}$ is dependent; if $z \notin J_x$, then $J_x \subseteq S \cup \{ x \}$, which means $S \cup \{ x \}$ is dependent. Hence $S \cup \{ x \}$ is always dependent for all $x \in E_{\pm n}-P \cup P^*$. Moreover, $S \cup \{ z \}$ is dependent.

We are left with $S \cup \{ z^* \}$. Suppose $S \cup \{ z^* \}$ is dependent. Then $S$ is maximally independent, and hence a basis. However, we have $|S| = k-2$, which contradicts $|B| = k$. Suppose $S \cup \{ z^* \}$ is independent. Then $S \cup \{ z^* \}$ is maximally independent, and hence a basis. However, we have $|S \cup \{ z^* \}| = k-1$, which contradicts $|B| = k$. Therefore, there exists no $P$ such that $|P| < |B|$ and $P$ spans $E_{\pm n}-P \cup P^*$.
\end{proof}

Below we state the \textit{Symmetric Exchange Axiom}. 

For every $X,Y \in \mathcal{B}$, if $i \in Y - X$, then there exists a $j \in X-Y$ such that $X \cup \{i\} - \{j\} \in \mathcal{B}$.

We show this Symmetric Exchange Axiom leads to the Maximality Property of symplectic matroids.

\begin{Th}
\label{theorem2}
If $\mathcal{B}$ is a collection of admissible sets of cardinality $k$ in $[n] \cup [n]^\ast$ where $k \leq n$, then the Symmetric Exchange Axiom guarantees the Maximality Property.
\end{Th}
\begin{proof}
This proof resembles that in \cite{borov}. Assume $\mathcal{B}$ satisfies the Symmetric Exchange Axiom. X and $X \cup \{i\} - \{j\}$ must be comparable because the ordering of $E_{\pm n}$ is total. Suppose X, Y are two distinct maximal bases. Let $i$ be the maximal element of $X \Delta Y$. Without the loss of generality, suppose $i \in Y$. Then there exists some $j \in X$ such that $X \cup \{i\}-\{j\} \in \mathcal{B}$. We know X and $X \cup \{i\} - \{j\}$ are comparable and distinct. Since $X$ is maximal and $i$ is the maximal element in $X \Delta Y$, then $X \cup \{i\} - \{j\}$ is greater than $X$. This causes a contradiction. Therefore, the Symmetric Exchange Axiom induces the Maximality Property.\\
\end{proof}

\section{Circuit Axioms leading to symplectic matroids}
\label{sec5}

Now we prove the other direction of the main theorem. Lemma~\ref{lemma1}, Lemma~\ref{lemma2}, and Theorem~\ref{theorem1} already told us that when $\mathcal{B}$ is the collection of bases of a symplectic matroid, then (SC1) - (SC4) hold. Now suppose $\mathcal{C}$ is a collection satisfying axioms (SC1) - (SC4) and $\mathcal{B}$ the collection of maximal admissible subsets of $E_{\pm n}$ not containing members of $\mathcal{C}$. We prove the following claims.

\noindent \textbf{Claim 1}\\
The bases in $\mathcal{B}$ are equi-numerous.
 
Suppose $B_1 , B_2 \in \mathcal{B}$ such that $|B_1| < |B_2|$. By Axiom (SC4), there exists an $x \in E_{\pm n} - B_1 \cup B_1^*$ that $B_1$ doesn't span. Then $B_1 \cup \{ x \}$ is admissible and meantime contains no circuit, which contradicts the maximality of the basis $B_1$.

\noindent \textbf{Claim 2}\\
Let $B \in \mathcal{B}$, and some $x \notin B$ such that $B \cup \{ x \}$ is admissible. Then there exists a unique circuit $C \in B \cup \{x\}$ where $C$ is given by 
\[C=\{x\} \cup \{b \in B \mid B \cup \{x\} - \{b\} \in \mathcal{B}\}. \]

To prove Claim 2, we let $x \notin B$ such that $B \cup \{ x \}$ is admissible. Then there exists some $D \in \mathcal{C}$ such that $D \subseteq B \cup \{ x \}$. If $\{ x \} \in \mathcal{C}$, then we are done. Otherwise let  
\[C=\{x\} \cup \{b \in B \mid B \cup \{x\} - \{b\} \in \mathcal{B}\}. \]
We want to show $C = D$.

Since $D \not \subset B$, we know $x \in D$. Now let $y \in D - \{ x \}$. Then $y \in B$. Let $A := B \cup \{ x \} - \{ y \}$. Suppose, for contradiction, that $A$ contains some circuit $E \in \mathcal{C}$. For sure $x \in E$. If $E$ and $D$ are distinct, then by Axiom (SC3), there exists some circuit $F$ such that $F \subseteq E \cup D - \{ x \}$. But then $F \subseteq B$, a contradiction. Hence $E = D$. However $y \notin E$, $y \in D$. Thus we reach a contradiciton.

Hence $A \in \mathcal{B}$. So $y \in C$ and $D \subseteq C$. To show $C \in \mathcal{C}$, we must show $C - \{ z\}$ is independent for all $z \in C$. If $z = x$, then $C - \{z \} \subseteq B \in \mathcal{B}$. Otherwise $C - \{ z \} \subseteq B \cup \{x \} - \{z \}$, which is a member of $\mathcal{B}$ by the definition of $C$. Therefore $C,D \in \mathcal{C}$ and by Axiom (SC2), we have $D = C$. Claim 2 is proved.

Let $A, B \in \mathcal{B}$ with $a \in A - B$. We show that there exists $b \in B-A$ such that $B \cup \{ a\} - \{ b\} \in \mathcal{B}$. Claim 2 says that there exists a circuit $C \in B \cup \{a\}$ such that 
\[C - \{a\} =  \{b \in B \mid B \cup \{a\} - \{b\} \in \mathcal{B}\}. \]
However, $C - \{a\}$ is never empty because otherwise $C \subseteq A$ and is thus independent, which leads to a contradiction. So the Symmetric Exchange Property is satisfied here, which leads to the Maximality of symplectic matroids by Theorem~\ref{theorem2}.

\section{From graphs to symplectic matroids}
\label{sec6}

In this section, a \textit{graph} refers to a finite undirected multigraph. Inspired by Theorem 2 in \cite{chow}, we apply Theorem~\ref{thm:main} to see how every graph gives rise to a symplectic matroid.

Let $G$ be a graph with $n$ edges $e_1, e_2, \ldots, e_n$. We define a family $\mathcal{C}(G)$ of admissible
subsets of $E_{\pm n}$ as follows. If $S \subseteq E_{\pm n}$ is admissible, let
\[ G(S) : = \{ e_i \mid i \in S \textrm{ or } i^* \in S \}.
\]
We let an admissible set $S$ be a member of $\mathcal{C}(G)$ if and only if 
\begin{enumerate}
\item either $G(S)$ is a (single) cycle and there is an even number of edges $e_i$ in $G(S)$ such that $i^* \in S$ (The parity of $G(S)$ is the product of the signs of these edges and is thus positive);
\item or $G(S)$ is a union of (single) cycles, there is an even number of edges $e_i$ in $G(S)$ such that $i^* \in S$, and in each cycle there is an odd number of edges with negative signs.
\end{enumerate}

We use some notions and terms from \cite{Zas}, which we review now. A \textit{signed graph} is a graph with each edge given either a plus sign or a minus sign. A
cycle in a signed graph is \textit{balanced} if the product of the signs of the corresponding edges is positive
and is \textit{unbalanced} otherwise. Every signed graph $\Gamma$ gives rise to an ordinary matroid
$M(\Gamma)$ in the following manner. The ground set of $M(\Gamma)$ is the signed edge
set of $\Gamma$, and a set of edges is independent if every connected
component is either a tree or a unicyclic graph whose unique cycle is unbalanced. \cite[Theorem~5.1]{Zas} shows that
$M(\Gamma)$ is a matroid. Notice that a basis of $M(\Gamma)$ can have as many elements as $G$ has vertices, but not more.

To phrase this another way, our construction of $\mathcal{C}(G)$ is the union of all $M(\Gamma)$ as $\Gamma$ ranges over all $2^n$ signed graphs with underlying graph $G$. 

\begin{Th}
For every graph $G$, $\mathcal{C}(G)$ is the collection of circuits of a symplectic matroid.
\end{Th}
\begin{Rem}
The symplectic matroid we construct from graph $G$ in terms of circuits is the same matroid constructed differently in terms of independent sets by Theorem 2 in \cite{chow}.
\end{Rem}

\begin{proof}
It is easy to check that members of $\mathcal{C}(G)$ satisfy (SC1) and (SC2).  Let $C_1 , C_2 \in \mathcal{C}(G)$ and suppose $C_1 , C_2$ are single cycles. If $C_1 \cap C_2 \neq \emptyset$ and suppose $e_1 \in C_1 \cap C_2$, there definitely exists a cycle $C_3 = C_1 \cup C_2 - C_1 \cap C_2 \subseteq C_1 \cup C_2 - \{ e_1 \}$. For any $e \in C_1 \cap C_2$, the deletion of such an edge doesn't change the parity of $C_3$ because we delete it twice from $C_1$ and $C_2$. Thus there is an even number of negative edges in $C_3$. If either or both of $C_1$ and $C_2$ are unions of (single) cycles, the proof would be analogous. 

\cite[Theorem~5.1]{Zas} shows that
$M(\Gamma)$ is a matroid whose set of edges is independent if every connected
component is either a tree or a unicyclic graph. Notice that a basis of $M(\Gamma)$  cannot have more elements than $G$ has vertices. Therefore if an admissible subset $P$ of $E_{\pm n}$ satisfies $|P| < |B|$, then $|P| < \# V(G) - 1$. In other words, $G(P)$ is a subset of a spanning tree in $G$. Therefore if $P$ spans $x \in E_{\pm n} - P \cup P^*$,  there exists a unique $J \in \mathcal{C}(G)$ such that $J - P = \{ x \}$. Considering the parity of $J$, $P$ is not able to span $x^*$ at the same time. Thus $P$ does not span $E_{\pm n} - P \cup P^*$.

Therefore $\mathcal{C}(G)$ is the collection of circuits of a symplectic matroid.
\end{proof}

\section{Acknowledgements} 

I wish to thank Ed Swartz for many helpful discussions and corrections.  I also thank Timothy Chow for his multiple suggestions, especially on rephrasing Theorem~\ref{thm:main}.

\bibliography{mybibfile}

\end{document}